\documentclass[a4paper,reqno]{amsart} 

\usepackage{amssymb}
\usepackage[mathcal]{euscript} 

\usepackage{color}

\newcommand{\bydef}{:=} 
\newcommand{\veps}{\varepsilon}
\newcommand{\ab}{\mathrm{ab}}
\newcommand{\id}{\mathrm{id}}
\newcommand{\wh}[1]{\widehat{#1}}
\newcommand{\vphi}{\varphi} 
\newcommand{\diag}{\mathrm{diag}}
\newcommand{\matr}[1]{\left(\begin{smallmatrix}#1\end{smallmatrix}\right)}
\newcommand{\tr}{\mathrm{tr}}
\newcommand{\zero}{{\bar{0}}} 
\newcommand{\one}{{\bar{1}}}

\newcommand{\cA}{\mathcal{A}}
\newcommand{\cB}{\mathcal{B}}
\newcommand{\cC}{\mathcal{C}}
\newcommand{\cD}{\mathcal{D}} 
\newcommand{\cH}{\mathcal{H}}
\newcommand{\cL}{\mathcal{L}}
\newcommand{\cM}{\mathcal{M}}
\newcommand{\cN}{\mathcal{N}} 
\newcommand{\cR}{\mathcal{R}} 
\newcommand{\cW}{\mathcal{W}}

\newcommand{\frg}{{\mathfrak g}}
\newcommand{\frh}{{\mathfrak h}}
\newcommand{\fra}{{\mathfrak a}}
\newcommand{\frs}{{\mathfrak s}}
\newcommand{\frX}{\mathfrak{X}}

\newcommand{\ZZ}{\mathbb{Z}}
\newcommand{\FF}{\mathbb{F}} 
\newcommand{\HH}{\mathbb{H}} 
\newcommand{\OO}{\mathbb{O}}
\newcommand{\QQ}{\mathbb{Q}}
\newcommand{\RR}{\mathbb{R}}
\newcommand{\Alb}{\mathbb{A}} 

\newcommand{\ad}{\mathrm{ad}}
\newcommand{\Gl}{\mathfrak{gl}} 
\newcommand{\Sl}{\mathfrak{sl}}
\newcommand{\frso}{{\mathfrak{so}}}

\DeclareMathOperator{\rank}{\mathrm{rank}} 
\DeclareMathOperator*{\ot}{\otimes}
\DeclareMathOperator{\Int}{\mathrm{Int}} 
\DeclareMathOperator{\trank}{\mathrm{tor.\!rank}} 
\DeclareMathOperator{\chr}{\mathrm{char}}

\DeclareMathOperator{\Aut}{\mathrm{Aut}}
\DeclareMathOperator{\Der}{\mathrm{Der}}
\DeclareMathOperator{\Stab}{\mathrm{Stab}}
\DeclareMathOperator{\Hom}{\mathrm{Hom}}
\DeclareMathOperator{\Diag}{\mathrm{Diag}}
\DeclareMathOperator{\End}{\mathrm{End}}
\DeclareMathOperator{\Cent}{\mathrm{Cent}}

\DeclareMathOperator{\AAut}{\mathbf{Aut}} 
\newcommand{\Diags}{\mathbf{Diag}}
\newcommand{\Stabs}{\mathbf{Stab}}
\newcommand{\Qs}{\mathbf{Q}}
\newcommand{\Ts}{\mathbf{T}}

\newenvironment{romanenumerate} 
{\begin{enumerate}

}{\end{enumerate}}

\newtheorem{theorem}{Theorem}[section]
\newtheorem{proposition}[theorem]{Proposition}
\newtheorem*{proposition*}{Proposition}
\newtheorem{lemma}[theorem]{Lemma}
\newtheorem{corollary}[theorem]{Corollary}

\theoremstyle{definition} 
\newtheorem{df}[theorem]{Definition}
\newtheorem{example}[theorem]{Example}

\theoremstyle{remark} \newtheorem{remark}[theorem]{Remark}
\numberwithin{equation}{section}

\begin{document}

\title[Almost fine gradings and classification of gradings]{Almost fine gradings on algebras and classification of gradings up to isomorphism}

\author[A.~Elduque]{Alberto Elduque} 
\address[A.E.]{Departamento de Matem\'{a}ticas e Instituto Universitario de Matem\'aticas y Aplicaciones, 
	Universidad de Zaragoza, 50009 Zaragoza, Spain}
\email{elduque@unizar.es} 
\thanks{A.E. is supported by grant PID2021-123461NB-C21, funded by MCIN/AEI/10.13039/\\ 501100011033 and 
 ``ERDF A way of making Europe'', and by grant E22\_23R (Gobierno de Arag\'on, Grupo de investigaci\'on ``\'Algebra y Geometr{\'\i}a'').}

\author[M.~Kochetov]{Mikhail Kochetov}
\address[M.K.]{Department of Mathematics and Statistics, 
	Memorial University of Newfoundland, St.~John's, NL, A1C5S7, Canada}
\email{mikhail@mun.ca}
\thanks{M.K. is supported by Discovery Grant 2018-04883 of the Natural Sciences and Engineering Research Council (NSERC) of Canada.}


\subjclass[2020]{Primary 17A01; Secondary 17A36, 17B70, 16W50, 20G07}

\keywords{Grading; classification; automorphism group; Lie algebra; root grading.} 

\begin{abstract}
We consider the problem of classifying gradings by groups on a finite-dimensional algebra $\cA$ (with any number of multilinear operations) over an algebraically closed field. We introduce a class of gradings, which we call almost fine, such that every $G$-grading on $\cA$ is obtained from an almost fine grading on $\cA$ in an essentially unique way, which is not the case with fine gradings. For abelian $G$, we give a method of obtaining all almost fine gradings if fine gradings are known. We apply these ideas to the case of semisimple Lie algebras in characteristic $0$: to any abelian group grading with nonzero identity component, we attach a (possibly nonreduced) root system $\Phi$ and, in the simple case, construct an adapted $\Phi$-grading.
\end{abstract}

\maketitle

\section{Introduction}\label{se:intro}

Gradings by groups on algebras appear in many areas of mathematics and mathematical physics. 
For example, if $G$ is any abelian group, then a $G$-grading on the algebra of polynomials (or Laurent polynomials) $\cA$ can be defined by assigning each variable a ``weight'' in $G$. 
We also obtain a $G$-grading on the Lie algebra of derivations $\Der(\cA)$ and on some of its important subalgebras. 
In fact, many Lie algebras come equipped with a grading by a free 
abelian group (for example, the root lattice in the case of complex 
semisimple Lie algebras), which plays a crucial role in their 
representation theory. These include the Lie algebras graded 
by not 
necessarily reduced root systems, or root-graded Lie algebras. Among 
these we find the $\ZZ$-graded Lie algebras attached to 
Jordan algebras and more general Jordan systems, or to structurable
algebras.
Gradings by $\ZZ/2\ZZ$ and more general finite abelian groups appear in the study of superalgebras, symmetric spaces, and Kac--Moody Lie (super)algebras.

Starting with \cite{PZ}, there has been considerable interest in describing all possible group gradings on important algebras, such as simple Lie algebras. In particular, abelian group gradings are closely related to symmetries (i.e., automorphisms) of the algebra.
Indeed, over an algebraically closed field of characteristic $0$, any grading by an abelian group $G$ on a finite-dimensional algebra $\cA$ 
is given by a homomorphism of algebraic groups $\widehat{G}\rightarrow \Aut(\cA)$, where 
$\widehat{G}$ is the group of multiplicative characters of the grading group $G$.  In particular, as pointed out in \cite{PZ}, the so-called \emph{fine gradings} (see below) on $\cA$ correspond to maximal abelian diagonalizable subgroups of 
$\Aut(\cA)$. Over arbitrary fields, this connection is preserved but one needs to 
use affine group schemes instead of linear algebraic groups (see e.g. \cite{EKmon}). 

In the view outlined above, the algebra $\cA$ takes center stage, while its grading groups are derived from it --- via the group (scheme) of automorphisms in the case of abelian group gradings.
On the other hand, if one wants to work with a category of graded algebras, it becomes necessary to fix the grading group $G$, which may also carry additional structure (such as a bicharacter or a cocycle) 
to define certain operations on the category (such as braiding). These two points of view lead to different kinds of classification of gradings: fine gradings up to \emph{equivalence} and $G$-gradings up to \emph{isomorphism} (see below).

The two kinds of classification are not independent of each other. Any $G$-grading on a finite-dimensional algebra is a coarsening of a fine grading. 
This means that any component of the $G$-grading is a sum of certain components of the fine grading. 
Actually, the $G$-grading is then determined by a homomorphism to $G$ from the \emph{universal group} (see \S\ref{sse:universal}) of the fine grading. 
However, there is no canonical way to attach a specific fine grading to a given $G$-grading. 
The purpose of this paper is to define a new class of gradings, which we call \emph{almost fine gradings} because they are not too far from being fine (see Proposition~\ref{prop:2}  and Theorem~\ref{th:3}). 
But, unlike fine gradings, they allow us to attach to any $G$-grading a canonical almost fine grading.
In this way, the classification of $G$-gradings up to isomorphism reduces to the classification of almost fine gradings up to equivalence and to
the determination of the \emph{Weyl groups} (see \S\ref{sse:equivalence}) of these almost fine gradings and their actions on the
universal groups (Theorem \ref{th:1}).

To explain our approach more precisely, we first need some definitions.
Let  $G$ be a group and let $\cA$ be an algebra with any number of multilinear operations. $\cA$ is 
said to be a \emph{$G$-graded algebra} if there is a fixed \emph{$G$-grading} on $\cA$, i.e., a 
direct sum decomposition of its underlying vector space, $\cA=\bigoplus_{g\in G}\cA_g$, such that, 
for any operation $\varphi$ defined on $\cA$, we have 
$\varphi(\cA_{g_1},\ldots,\cA_{g_n})\subset\cA_{g_1\cdots g_n}$ for all $g_1,\ldots,g_n\in G$, 
where $n$ is the number of arguments taken by $\varphi$.  The subspaces $\cA_g$ are called
\emph{homogeneous components}. For any nonzero element $a\in\cA_g$, we will say that $a$ is
\emph{homogeneous of degree $g$} and write $\deg a=g$. (The zero vector is also considered
homogeneous, but its degree is undefined.)

For a fixed group $G$, the class of $G$-graded vector spaces is a category in which morphisms are 
the linear maps that preserve degree. In particular, we can speak of isomorphism of $G$-graded 
algebras. Two $G$-gradings, $\Gamma$ and $\Gamma'$, on the same algebra $\cA$ are said to 
be \emph{isomorphic} if there exists an isomorphism of $G$-graded algebras 
$(\cA,\Gamma)\to(\cA,\Gamma')$ or, in other words, there exists an automorphism of the algebra 
$\cA$ that maps each component of $\Gamma$ onto the component of $\Gamma'$ of the same degree. 
If the automorphism maps each component of $\Gamma$ onto a component of $\Gamma'$, but
 not necessarily of the same degree, then $\Gamma$ and $\Gamma'$ are said to be \emph{equivalent}; 
in this setting the group $G$ need not be fixed. A grading is said to be \emph{fine} if it has no 
proper refinement (see \S\ref{sse:fine}). 

As already mentioned, one may wish to classify all $G$-gradings on $\cA$ up to isomorphism or 
fine gradings on $\cA$ up to equivalence. Both problems received much attention in the last 
two decades, especially for simple algebras in many varieties: associative, associative with involution, 
Lie, Jordan, alternative, various triple systems, and so on (see, e.g., \cite{EKmon} and the 
references therein, also \cite{Ara,BKR_Lie,AC,DET,EKR}). It should be noted that a solution 
to one of these 
problems is often instrumental in solving the other, but not in a straightforward way. 

Any group homomorphism $\alpha\colon G\to H$ gives a functor from $G$-graded vector spaces to 
$H$-graded ones: for $V=\bigoplus_{g\in G}V_g$, we define the $H$-graded vector space 
${}^\alpha V$ to be the same space $V$ but equipped with the $H$-grading 
$V=\bigoplus_{h\in H} V'_h$ where $V'_h\bydef\bigoplus_{g\in\alpha^{-1}(h)}V_g$. 
(This functor is the identity map on morphisms.) 
If the $G$-grading on $V$ is denoted by $\Gamma$, the corresponding $H$-grading on $V$ will be 
called \emph{induced by $\alpha$} and denoted ${}^\alpha\Gamma$. Note that the 
homogeneous elements of degree $g$ with respect to $\Gamma$ become homogeneous of degree 
$\alpha(g)$ with respect to $^\alpha{}\Gamma$.

If $\{\Gamma_i\}_{i\in I}$ is a set of representatives of the equivalence classes of 
fine gradings on a finite-dimensional algebra $\cA$ and $U_i$ is the universal group of $\Gamma_i$, then any $G$-grading $\Gamma$ on $\cA$ is isomorphic to the 
induced grading ${}^\alpha\Gamma_i$ for some $i\in I$ and a group homomorphism 
$\alpha\colon U_i\to G$. However, both $i$ and $\alpha$ are usually far from unique, so we do not easily 
obtain a classification of $G$-gradings up to isomorphism. In this paper, we will show how to extend the class of fine gradings 
and at the same time restrict the homomorphisms $\alpha$ to obtain uniqueness (up to the action of the Weyl group).
This approach may be applied when the fine gradings on a certain algebra are known and one wishes to classify $G$-gradings. 
For example, this is the case for the exceptional simple Lie algebras of types $E_6$, $E_7$, $E_8$ (see \cite{EKmon,Eld16} and the references therein). 

The paper is structured as follows. After reviewing preliminaries on gradings and algebraic groups in Section~\ref{se:preliminaries}, 
we introduce almost fine gradings on a finite-dimensional algebra $\cA$ in 
Section~\ref{se:def_ex}. The goal of Section~\ref{se:classification_iso} is to prove Theorem~\ref{th:1}, which classifies all $G$-gradings on $\cA$ up to isomorphism if we know almost fine gradings on $\cA$ up 
to equivalence. In Section~\ref{se:fine_to_almost_fine}, we discuss how to obtain almost fine gradings if we know fine gradings  (Proposition~\ref{prop:2}  and Theorem~\ref{th:3}). Finally,
 in Section~\ref{se:Lss}, we apply these ideas to the case of abelian group gradings on semisimple
 Lie algebras that have nontrivial identity component: to any such grading $\Gamma$ on $\cL$, we 
attach a (possibly nonreduced) root system $\Phi$ (Theorem~\ref{th:Lss_rootsystem}) and, in 
the case of simple $\cL$, construct a $\Phi$-grading on $\cL$ adapted to 
$\Gamma$ (Theorem~\ref{th:Lss_structure}).

Except in Section~\ref{se:preliminaries}, we assume that the ground field $\FF$ is 
\emph{algebraically closed}. The characteristic is arbitrary unless stated otherwise.

\section{Preliminaries on gradings}\label{se:preliminaries}
 
In this section we will briefly review some general facts and terminology concerning gradings on 
algebras, most of which go back to J.~Patera and H.~Zassenhaus \cite{PZ}. 
We will also introduce notation that is used throughout the paper. The reader is referred to Chapter~1 
of the monograph \cite{EKmon} for more details, and to \cite{Hum_lag,Wat} for the background 
on (linear) algebraic groups and (affine) group schemes.
 
 \subsection{Gradings and their universal groups}\label{sse:universal}
 
There is a more general concept of a grading on an algebra $\cA$, namely, a set 
of nonzero subspaces of $\cA$, which we write as $\Gamma=\{\cA_s\}_{s\in S}$ for convenience, 
such that $\cA=\bigoplus_{s\in S}\cA_s$ and, for any $n$-ary operation $\varphi$ defined on 
$\cA$ and any $s_1,\ldots,s_n\in S$, there exists $s\in S$ such that 
$\varphi(\cA_{s_1},\ldots,\cA_{s_n})\subset\cA_s$. Any $G$-grading on $\cA$ becomes a 
grading in this sense if we take $S$ to be its \emph{support}: $S=\{g\in G\mid\cA_g\ne 0\}$. 
 
 For a given grading $\Gamma$ on an algebra $\cA$ as above, there may or may not exist a realization 
of $\Gamma$ over a group $G$, by which we mean an injective map $\iota\colon S\to G$ such that assigning 
the nonzero elements of $\cA_s$ degree $\iota(s)\in G$, for all $s\in S$, and taking the 
homogeneous components of degree in $G\smallsetminus\iota(S)$ to be zero defines a 
$G$-grading on $\cA$. If such realizations exist, there is a universal one among them. Indeed, 
let $U=U(\Gamma)$ be the group generated by the set $S$ subject to all relations of the form 
$s_1\cdots s_n=s$ whenever $0\ne\varphi(\cA_{s_1},\ldots,\cA_{s_n})\subset\cA_s$ for an 
$n$-ary operation $\varphi$ on $\cA$. It is easy to see that $\Gamma$ admits a realization over a group (not 
necessarily $U$) if and only if the canonical map $\iota_0\colon S\to U$ is injective. If this is the case, then we will say that 
$\Gamma$ is a \emph{group grading}. Then $(U,\iota_0)$ is universal among all realizations 
$(G,\iota)$ of $\Gamma$ in the sense that there exists a unique group homomorphism 
$\alpha\colon U\to G$ such that $\alpha\iota_0=\iota$. We will call $U$ the \emph{universal group} of 
$\Gamma$. 
 
 Many gradings (for example, all group gradings on simple Lie algebras) can be realized over an 
abelian group. We will call them \emph{abelian group gradings}. Let $U_\ab=U_\ab(\Gamma)$ be 
the abelianization of $U(\Gamma)$, i.e., the \emph{abelian} group generated by $S$ subject to
 the relations above. Then $\Gamma$ has a realization over some abelian group if and only if 
the canonical map $\iota_0\colon S\to U_\ab$ is injective, and in this case $(U_\ab,\iota_0)$ is the 
universal one among such realizations. We will call $U_\ab$ the \emph{universal abelian group} of 
$\Gamma$.
 
 \subsection{Equivalence and automorphisms of gradings}\label{sse:equivalence}
 
 An \emph{equivalence of graded algebras} from $\cA=\bigoplus_{s\in S}\cA_s$ to 
$\cB=\bigoplus_{t\in T}\cB_t$ is an algebra isomorphism $\psi\colon \cA\to\cB$ such that, for any 
$s\in S$, we have $\psi(\cA_s)=\cB_t$ for some $t\in T$. Since we assume that all $\cA_s$ are nonzero, 
$\psi$ defines a bijection $\gamma\colon S\to T$ such that $\psi(\cA_s)=\cB_{\gamma(s)}$ for all $s\in S$. 
If these are group gradings and we realize them over their universal groups, $U$ and $U'$, then 
any equivalence $\psi\colon \cA\to\cB$ leads to an isomorphism: the bijection of the supports 
$\gamma\colon S\to T$ determined by $\psi$ extends to a unique isomorphism of the universal groups, 
which we also denote by $\gamma$, so that $\psi\colon {}^\gamma\cA\to\cB$ is an isomorphism of 
$U'$-graded algebras.
 
 Two gradings, $\Gamma$ and $\Gamma'$, on the same algebra $\cA$ are said to be \emph{equivalent} 
if there exists an equivalence $(\cA,\Gamma)\to(\cA,\Gamma')$ or, in other words, there exists 
an automorphism of the algebra $\cA$ that maps the set of nonzero components of $\Gamma$ to 
that of $\Gamma'$. 
 In particular, we can consider the group $\Aut(\Gamma)$ of all equivalences from the graded algebra 
$(\cA,\Gamma)$ to itself. Applying the above property of universal groups, we see that the permutation 
of the support of $\Gamma$ defined by any element of $\Aut(\Gamma)$ extends to a 
unique automorphism of the universal group $U=U(\Gamma)$. This gives us a group homomorphism 
$\Aut(\Gamma)\to\Aut(U)$, whose kernel is denoted $\Stab(\Gamma) $ and consists of all 
degree-preserving automorphisms, i.e., isomorphisms from the graded algebra $(\cA,\Gamma)$ to 
itself. The image of this homomorphism $\Aut(\Gamma)\to\Aut(U)$ is known as the \emph{Weyl group} 
of the grading $\Gamma$:
 \[
 W(\Gamma)\bydef\Aut(\Gamma)/\Stab(\Gamma)\hookrightarrow\Aut(U(\Gamma)).
 \]
 If we deal with abelian group gradings, the universal abelian groups can be used, and we can 
regard $W(\Gamma)$ as a subgroup of $\Aut(U_\ab(\Gamma))$.
 
 \subsection{Fine gradings}\label{sse:fine}
 
 A grading $\Gamma:\cA=\bigoplus_{s\in S}\cA_s$ is said to be a \emph{refinement} of a grading 
$\Gamma':\cA=\bigoplus_{t\in T}\cA'_t$ (or $\Gamma'$ a \emph{coarsening} of $\Gamma$) if, for 
any $s\in S$, there exists $t\in T$ such that $\cA_s\subset\cA'_t$. If the inclusion is proper for at 
least one $s\in S$, the refinement (or coarsening) is called \emph{proper}. 
 
 For example, if $\Gamma$ is a grading by a group $G$ (so $S\subset G$) and $\alpha\colon G\to H$ is a 
group homomorphism, then ${}^\alpha\Gamma$ is a coarsening of $\Gamma$, which is proper if and 
only if $\alpha|_S$ is not injective. If $G$ is the universal group of $\Gamma$, then any coarsening 
$\Gamma'$ that is a grading by a group $H$ necessarily has the form ${}^\alpha\Gamma$ for a 
unique group homomorphism $\alpha\colon G\to H$. 
 
 A group grading (respectively, abelian group grading) is said to be \emph{fine} if it does not have 
a proper refinement that is itself a group (respectively, abelian group) grading. Note that the concept 
of fine grading is relative to the class that we consider. For example, there is a  $\ZZ^n$-grading 
defined on the matrix algebra $M_n(\FF)$ by declaring the degree of the matrix unit $E_{ij}$ to be 
$\veps_i-\veps_j$, where $\{\veps_1,\ldots,\veps_n\}$ is the standard basis of $\ZZ^n$. This grading 
is fine in the class of group gradings, but has a refinement whose components are the 
one-dimensional subspaces spanned by $E_{ij}$. This latter cannot be realized over a group, although 
it can be realized over a semigroup, for instance, over the semigroup $\{0,\veps_{ij}\mid  1\leq i,j\leq n\}$
with $0\veps_{ij}=\veps_{ij}0=0^2=0$ and $\veps_{ij}\veps_{kl}=\delta_{jk}\veps_{il}$, 
by assigning degree $\veps_{ij}$ to the matrix unit $E_{ij}$. (Note that the group completion 
of this semigroup is trivial.)
 
 \subsection{Gradings and actions}\label{sse:grad_act}
 
 Given a $G$-grading $\Gamma:\cA=\bigoplus_{g\in G}\cA_g$, any group homomorphism 
$\chi\colon G\to\FF^\times$, where $\FF^\times$ denotes the multiplicative group of $\FF$, acts as 
an automorphism of $\cA$ as follows: $\chi\cdot a=\chi(g)a$ for all $a\in\cA_g$ and $g\in G$, which 
is then extended to the whole $\cA$ by linearity. Note that this is actually an automorphism of $\cA$ 
as a graded algebra, as it leaves each component $\cA_g$ invariant --- in fact, acts on it as the 
scalar operator $\chi(g)\,\id_{\cA_g}$. Thus $\Gamma$ defines a group homomorphism 
$\eta_\Gamma$ from the group of (multiplicative) characters $\wh{G}\bydef\Hom(G,\FF^\times)$ 
to the automorphism group $\Aut(\cA)$, which is particularly useful if $G$ is abelian and $\FF$ 
is algebraically closed and of characteristic $0$, because then $\wh{G}$ separates points of $G$ 
and, therefore, the grading $\Gamma$ can be recovered as a simultaneous eigenspace decomposition 
with respect to these automorphisms:
\begin{equation}\label{eq:eigenspace}
 \cA_g=\{a\in\cA\mid\chi\cdot a=\chi(g)a\text{ for all }\chi\in\wh{G}\}.
\end{equation}
 For example, the above $\ZZ^n$-grading on $M_n(\FF)$ corresponds to the homomorphism from 
the (algebraic) torus $(\FF^\times)^n$ to $\Aut(M_n(\FF))$ that sends $(\lambda_1,\ldots,\lambda_n)$ 
to the inner automorphism $\Int\diag(\lambda_1,\ldots,\lambda_n)$. 
 
 If $\cA$ is finite-dimensional and $\FF$ is algebraically closed, then $\Aut(\cA)$ is an algebraic group 
(see, e.g., \cite[Exercise 7.3]{Hum_lag} or \cite[\S 7.6]{Wat}). If $G$ is a finitely generated 
 abelian group, then $\wh{G}$ is a \emph{diagonalizable} algebraic group (\cite[\S 16]{Hum_lag}),
 isomorphic to the 
 direct product of a torus and a finite abelian group whose order is not divisible by $\chr\FF$; 
such groups are often called  \emph{quasitori} (especially in characteristic $0$). For any $G$-grading on 
$\cA$, $\eta_\Gamma$ is a homomorphism of algebraic groups. Conversely, 
 the image of any homomorphism of algebraic groups $\eta\colon \wh{G}\to\Aut(\cA)$ consists of 
commuting diagonalizable operators and, therefore, defines a simultaneous eigenspace decomposition of 
 $\cA$ indexed by the homomorphisms of algebraic groups $\wh{G}\to\FF^\times$, which are 
 canonically identified with the elements of $G$ if $\chr\FF=0$ or $\chr\FF=p$ and $G$ has no $p$-torsion. 
Thus the subspaces $\cA_g$ defined by \eqref{eq:eigenspace}, with respect to the $\wh{G}$-action 
$\chi\cdot a=\eta(\chi)a$, form a 
$G$-grading $\Gamma$ on the algebra $\cA$, and $\eta=\eta_\Gamma$. 
 
 If $\FF$ is not necessarily algebraically closed or $\chr\FF\ne 0$, one can recover the above 
one-to-one correspondence by using group schemes over $\FF$, namely, the automorphism group scheme 
$\AAut(\cA)$, defined by $\AAut(\cA)(\cR)\bydef\Aut_\cR(\cA\ot_\FF\cR)$ for any commutative 
associative unital $\FF$-algebra $\cR$, and the Cartier dual $G^D$ of an abelian group $G$, defined 
by $G^D(\cR)\bydef\Hom(G,\cR^\times)$. Then a $G$-grading $\Gamma$ on $\cA$ corresponds to 
the homomorphism of group schemes $\eta_\Gamma\colon G^D\to\AAut(\cA)$ defined by
 \begin{equation}\label{eq:eta}
 (\eta_\Gamma)_\cR(\chi)\colon a\ot r\mapsto a\ot\chi(g)r\text{ for 
     all }\chi\in\Hom(G,\cR^\times),\,a\in\cA_g,\,g\in G,\,r\in\cR.
 \end{equation}
The homomorphism $\eta_\Gamma\colon \wh{G}\to\Aut(\cA)$ in the previous paragraph is obtained by 
applying this one to $\FF$-points, i.e., taking $\cR=\FF$. 

 The image of the homomorphism $\eta_\Gamma\colon G^D\to\AAut(\cA)$ is contained in the 
following diagonalizable subgroupscheme $\Diags(\Gamma)$ of $\AAut(\cA)$, which can be defined for 
any grading $\cA=\bigoplus_{s\in S}\cA_s$:
 \begin{equation}\label{eq:Diag}
 \Diags(\Gamma)(\cR)\bydef\{\psi\in\Aut_\cR(\cA\ot\cR)\mid\psi|_{\cA_s\ot\cR}\in\cR^\times\,
\id_{\cA_s\ot\cR}\text{ for all }s\in S\}.
 \end{equation}
If $\Gamma$ is an abelian group grading and we take $G=U_\ab(\Gamma)$ in \eqref{eq:eta}, then 
it follows from the defining relations of $U_\ab(\Gamma)$ that 
$\eta_\Gamma\colon U_\ab(\Gamma)^D\to\Diags(\Gamma)$ is an isomorphism. In particular, the group of 
$\FF$-points $\Diag(\Gamma)$ is isomorphic to the group of characters of $U_\ab(\Gamma)$. 
Moreover, the automorphism group scheme $\Stabs(\Gamma)\bydef\AAut(\cA,\Gamma)$ coincides with 
the centralizer of $\Diags(\Gamma)$ in $\AAut(\cA)$.

Since we are going to assume that $\FF$ is algebraically closed, the group schemes that 
are \emph{algebraic} and \emph{smooth} can be identified with algebraic groups, by assigning to 
such a group scheme its group of $\FF$-points. For a finite-dimensional algebra $\cA$, $\AAut(\cA)$ 
is algebraic and, for an abelian group $G$, $G^D$ is algebraic if and only if $G$ is finitely generated. 
The smoothness condition is automatic if $\chr\FF=0$, but not so if $\chr\FF=p$. In particular, $G^D$
 is smooth if and only if $G$ has no $p$-torsion, and $\AAut(\cA)$ is smooth if and only if the tangent 
Lie algebra of the algebraic group $\Aut(\cA)$ coincides with $\Der(\cA)$, which is the tangent Lie 
algebra of the group scheme $\AAut(\cA)$ (in general, the former is contained in the latter). In any 
case, $G^D$ is a diagonalizable group scheme, and the centralizers of diagonalizable subgroupschemes 
in smooth group schemes are known to be smooth (\cite[Exp.~XI, 2.4]{SGA3}, cf. 
\cite[\S 18.4]{Hum_lag}). Hence, if $\Gamma$ is an abelian group grading on $\cA$ and 
$\AAut(\cA)$ is smooth, then so is $\Stabs(\Gamma)$.

\begin{proposition}\label{prop:no_p_torsion}
Let $\cA$ be a finite-dimensional algebra over an algebraically closed field $\FF$ such that 
$\AAut(\cA)$ is smooth. If $\chr\FF=p$, then $U_\ab(\Gamma)$ has no $p$-torsion for any fine 
abelian group grading $\Gamma$ on $\cA$.
\end{proposition}

\begin{proof}
This is equivalent to the statement that any maximal diagonalizable subgroupscheme $\Qs$ of 
$\AAut(\cA)$ is smooth. We have $\Qs\simeq G^D$ for some finitely generated abelian group $G$. 
We can write $G$ as the direct product of a finite $p$-group, a finite group of order coprime to $p$, 
and a free abelian group. Consider the corresponding decomposition $\Qs=\Qs_0\times\Qs_1\times\Ts$. 
Let $\mathbf{C}$ be the centralizer of $\Qs$ in $\AAut(\cA)$. Then $\Qs\subset\mathbf{C}$, 
$\mathbf{C}$ is smooth, and $\Ts$ is a maximal torus in $\mathbf{C}$. Indeed, if 
$\Ts\subset\Ts'\subset\mathbf{C}$ for some torus $\Ts'$, then $\Qs\subset\Qs\Ts'$ and $\Qs\Ts'$ 
is diagonalizable (as a homomorphic image of $\Qs\times\Ts'$), so we get $\Ts'\subset\Qs$ by 
maximality of $\Qs$, but then $\Ts=\Ts'$ since $\Ts$ is a maximal torus in $\Qs$. Consider the 
connected component $\mathbf{C}^\circ$ (see e.g. \cite[\S 6.7]{Wat}). It is smooth and contains 
$\Ts$ as its maximal torus. Since $\Ts$ is central in $\mathbf{C}^\circ$, $\mathbf{C}^\circ$ must 
be nilpotent (for example, apply \cite[\S 21.4]{Hum_lag} to the groups of $\FF$-points) and, therefore, 
$\mathbf{C}^\circ=\Ts\times\mathbf{U}$ where $\mathbf{U}$ is unipotent (see e.g. 
\cite[\S 10.4]{Wat}). Now, $\Qs_0$ is connected, so it is contained in $\mathbf{C}^\circ$. But 
its projection to $\mathbf{U}$ must be trivial, since $\mathbf{U}$ does not have nontrivial 
diagonalizable subgroupschemes. Therefore, $\Qs_0\subset\Ts$, which forces $\Qs_0=\mathbf{1}$.
\end{proof}

\begin{corollary}\label{cor:eigen} 
Any fine abelian group grading on $\cA$ is obtained as the eigenspace decomposition with respect to 
a unique maximal diagonalizable subgroup of $\Aut(\cA)$, namely, $\Diag(\Gamma)$.
\end{corollary}

Thus, if $\AAut(\cA)$ is smooth, then we have a one-to-one correspondence between the 
equivalence classes of fine abelian group gradings on $\cA$ and the conjugacy classes of 
maximal diagonalizable subgroups of $\Aut(\cA)$. 

\section{Definition and construction of almost fine gradings}\label{se:def_ex}

Let $\cA$ be a finite-dimensional algebra over an algebraically closed field $\FF$. Let 
$\Gamma:\cA=\bigoplus_{s\in S}\cA_s$ be a grading on $\cA$ with nonzero homogeneous components 
$\cA_s$, universal group $U$ and universal abelian group $U_\ab$ (see Subsection \ref{sse:universal}).

\subsection{Toral rank}

It is well known that, over an algebraically closed field, all maximal tori in an algebraic group 
are conjugate (see, e.g., \cite[\S 21.3]{Hum_lag}). In particular, they have the same dimension, which is known 
as the (reductive) rank of the algebraic group. Let $r$ be the rank of the automorphism group 
$\Aut(\cA)$.

\begin{df}
The rank of the algebraic group $\Stab(\Gamma)$ of the automorphisms of the graded algebra $(\cA,\Gamma)$ will be called the \emph{toral rank} of $\Gamma$ and denoted $\trank(\Gamma)$.
\end{df}

Since $\Stab(\Gamma)\subset\Aut(\cA)$, we have $0\le\trank(\Gamma)\le r$. A maximal torus of 
$\Aut(\cA)$ gives a $\ZZ^r$-grading on $\cA$, called its \emph{Cartan grading} (for example, the 
root space decomposition of a semisimple complex Lie algebra), whose toral rank is equal to $r$, since 
$\Stab(\Gamma)$ contains this maximal torus (see Subsection \ref{sse:grad_act}). 

If $\Gamma'$ is a coarsening of $\Gamma$ (see Subsection \ref{sse:fine}), then 
$\Stab(\Gamma)\subset\Stab(\Gamma')$, so $\trank(\Gamma)\le\trank(\Gamma')$. In particular, 
any coarsening of the Cartan grading has toral rank $r$. Those among the coarsenings that are 
themselves abelian group gradings are known as \emph{toral gradings}.

\subsection{Almost fine gradings}\label{sse:almost_fine}

For any grading $\Gamma$, $\Stab(\Gamma)$ contains the quasitorus 
\[
\Diag(\Gamma)\bydef\{\psi\in\Aut(\cA)\mid\psi|_{\cA_s}\in\FF^\times\,\id_{\cA_s}
\text{ for all }s\in S\}
\]
in its center. This quasitorus is isomorphic to the group of characters of the finitely generated 
abelian group $U_\ab=U_\ab(\Gamma)$, so its dimension is equal to the (free) rank of $U_\ab$, i.e., 
the rank of the free abelian group $U_\ab/t(U_\ab)$, where $t(U_\ab)$ denotes the torsion subgroup 
of $U_\ab$. Indeed, for any finitely generated abelian group $A$, the closed subgroup of 
$\wh{A}$ consisting of all characters of $A$ that kill $t(A)$ can be identified with the group of 
characters of $A/t(A)$, so it is a torus of dimension $\rank(A)$. The quotient of $\wh{A}$ by this 
subgroup can be identified with the group of characters of $t(A)$, so it is a finite abelian group 
(whose order is not divisible by $\chr\FF$). Therefore, the connected component of the identity 
$\Diag(\Gamma)^\circ$ is isomorphic to the group of characters of $U_\ab/t(U_\ab)$, a torus of 
dimension $\rank(U_\ab)$. In particular,
\[
\rank(U_\ab(\Gamma))\le\trank(\Gamma).
\]

\begin{df}\label{df:almost_fine}
A grading $\Gamma$ on $\cA$ is \emph{almost fine} if $\rank(U_\ab(\Gamma))=\trank(\Gamma)$ or, 
in other words, $\Diag(\Gamma)^\circ$ is a maximal torus in $\Stab(\Gamma)$.
\end{df}

For example, if $\Gamma$ has toral rank $0$, then $\Gamma$ is almost fine and 
$U_\ab(\Gamma)$ is finite.
Unlike fine gradings, almost fine gradings can have proper refinements, but at least the 
toral rank cannot drop:

\begin{proposition}\label{re:refinement}
 If $\Gamma$ is almost fine, then any refinement of $\Gamma$ is almost fine and has the same 
toral rank.
\end{proposition}

\begin{proof}
	If $\Gamma'$ is a refinement of $\Gamma$, then we have 
	\[
	\Diag(\Gamma)\subset\Diag(\Gamma')\subset\Stab(\Gamma')\subset\Stab(\Gamma).
	\]
	By hypothesis, $\Diag(\Gamma)^\circ$ is a maximal torus in $\Stab(\Gamma)$ and, hence, in 
$\Stab(\Gamma')$. 
	But $\Diag(\Gamma')^\circ$ is a torus and contains $\Diag(\Gamma)^\circ$, so 
$\Diag(\Gamma')^\circ=\Diag(\Gamma)^\circ$ by maximality. The result follows.
\end{proof}

In the next subsection, we will see that any grading $\Gamma$ admits an almost fine refinement 
and, moreover, if  $\Gamma$ is a group (respectively, abelian group) grading, then so is this 
refinement. Therefore, any fine grading (in the class of all gradings, group gradings, or abelian 
group gradings) is almost fine.

\begin{lemma}\label{re:stab0}
	If $\Gamma$ is almost fine, then $\Stab(\Gamma)^\circ$ is the direct product of the torus 
$\Diag(\Gamma)^\circ$ and a connected unipotent group (the unipotent radical).
\end{lemma}

\begin{proof}
	Since $\Diag(\Gamma)^\circ$ is central in $\Stab(\Gamma)^\circ$ and is a maximal torus
 by hypothesis, the connected algebraic group $\Stab(\Gamma)^\circ$ is nilpotent (see e.g. 
\cite[\S 21.4]{Hum_lag}) and, hence, the direct product of its (unique) maximal torus and 
unipotent radical (see e.g. \cite[\S 19.2]{Hum_lag}).
\end{proof}

In characteristic $0$, almost fine gradings can be characterized in terms of derivations of $\cA$. Let 
$\cD=\Der(\cA)$, which is the Lie algebra of the algebraic group $\Aut(\cA)$, and let
\begin{equation}\label{eq:def_L_e}
\cD_e=\{\delta\in\Der(\cA)\mid\delta(\cA_s)\subset\cA_s\text{ for all }s\in S\},
\end{equation}
which is the Lie algebra of $\Stab(\Gamma)$. Note that, if $\Gamma$ can be realized as a $G$-grading 
for a group $G$, then the associative algebra $\End(\cA)$ has an induced $G$-grading with the 
following components:
\[
\End(\cA)_g\bydef\{f\in\End(\cA)\mid f(\cA_h)\subset\cA_{gh}\text{ for all }h\in G\},
\]
and $\cD_e=\cD\cap\End(\cA)_e$, where $e$ denotes the identity element of $G$. Moreover, if $G$ 
is abelian, then $\cD$ has an induced $G$-grading: $\cD=\bigoplus_{g\in G}\cD_g$ where 
$\cD_g=\cD\cap\End(\cA)_g$.

\begin{proposition}
	Assume $\chr\FF=0$. A grading $\Gamma:\cA=\bigoplus_{s\in S}\cA_s$ is almost fine if and only 
if, for any element $\delta\in\cD_e$, each of the restrictions $\delta|_{\cA_s}$, $s\in S$, has a 
unique eigenvalue.
\end{proposition}

\begin{proof}
	If $\Gamma$ is almost fine, then the decomposition 
$\Stab(\Gamma)^\circ=\Diag(\Gamma)^\circ\times R_u$ of Lemma \ref{re:stab0} gives 
$\cD_e=\cM\oplus\cN$ where any element of $\cM$ acts as a scalar on each $\cA_s$ and any element of 
$\cN$ is nilpotent. It follows that any element of $\cD_e$ has a unique eigenvalue on each
 $\cA_s$. Conversely, suppose $\Diag(\Gamma)^\circ\subset T\subset\Stab(\Gamma)$ where $T$ is 
a torus. Then every element of the Lie algebra of $T$ is semisimple, so it must act as a scalar on each 
$\cA_s$ by hypothesis. This implies $T=\Diag(\Gamma)^\circ$, proving the maximality of 
$\Diag(\Gamma)^\circ$.
\end{proof}

There is a simpler characterization in the case of abelian group gradings if we assume that $\Aut(\cA)$ 
is a reductive algebraic group, by which we mean that its unipotent radical is trivial, but do not
 assume connectedness (contrary to the convention in \cite{Hum_lag}).  
It turns out that this characterization holds in prime characteristic as well if we assume that the 
group scheme $\AAut(\cA)$ is reductive, by which we mean that it is smooth and its group of 
$\FF$-points, $\Aut(\cA)$, is reductive.
The following result is probably known, but we could not find a reference:

\begin{lemma}\label{lm:red_cent}
	Let $G$ be a reductive algebraic group over an algebraically closed field. Then the centralizer 
$\Cent_G(Q)$ of any diagonalizable subgroup $Q\subset G$ is reductive.
\end{lemma}

\begin{proof}
	Without loss of generality, we assume that $Q$ is Zariski closed, so it is the product of a torus and 
a finite abelian group whose order is not divisible by $\chr\FF$. Since $\Cent_{G^\circ}(Q)$ has 
finite index in $\Cent_{G}(Q)$, it suffices to prove that $\Cent_{G^\circ}(Q)$ is reductive.
	Let $H$ and $Z$ be, respectively, the derived group $[G^\circ,G^\circ]$ and the connected 
component of the center $Z(G^\circ)^\circ$. Then  $Z$ is a torus and the radical of $G^\circ$, 
$H\cap Z$ is finite (see e.g. \cite[\S 19.5]{Hum_lag}), $H$ is connected semisimple, and  
$G^\circ=HZ$ (see e.g. \cite[\S 27.5]{Hum_lag}). 
	We claim that it suffices to prove that $\Cent_H(Q)$ is reductive. Indeed, let $C_H=\Cent_H(Q)$ 
and $C_Z=\Cent_Z(Q)$. Then
	\[
	C_H C_Z\subset\Cent_{G^\circ}(Q)\subset\tilde{C}_H\tilde{C}_Z,
	\]
	where $\tilde{C}_H\bydef\{h\in H\mid [h,q]\in H\cap Z\;\forall q\in Q\}$ and similarly for 
$\tilde{C}_Z$. Since $H\cap Z$ is finite, we have $\tilde{C}_H^\circ\subset C_H$ and 
$\tilde{C}_Z^\circ\subset C_Z$, so $C_H C_Z$ has finite index in $\tilde{C}_H\tilde{C}_Z$, and 
the claim follows.
	
	Replacing $G$ by $H$ and $Q$ by its image in the automorphism group of $H$ (see e.g. 
\cite[\S 27.4]{Hum_lag}), we arrive at the following setting: $G$ is a connected semisimple 
algebraic group, $Q$ is a closed diagonalizable subgroup of $\Aut(G)$, and we have to prove 
that the group of fixed points $G^Q$ is reductive. 
	Now, $Q$ defines a grading on the Lie algebra $\frg$ of $G$, and the Lie algebra $\frg_0$ of 
$G^Q$ is the identity component of this grading. 
	If $\chr\FF=0$, then a standard argument shows that the restriction of the Killing form of $\frg$ 
to $\frg_0$ is nondegenerate and, hence, $G^Q$ is reductive (see e.g. 
\cite[Prop.~6.2 of Chap.~1, Prop.~3.6 of Chap.~3]{OVG}, cf. \cite[Lemma~6.9]{EKmon}).
	Unfortunately, this approach does not work if $\chr\FF=p$, so we will make some further reductions. 
	
	First, we may suppose that $Q$ is finite (of order not divisible by $\chr\FF$), because $Q^\circ$ is 
a torus and the centralizers of tori in reductive algebraic groups are reductive 
(see e.g. \cite[\S 26.2]{Hum_lag}). By induction on $|Q|$, we may further suppose that 
$Q$ is cyclic: $Q=\langle s\rangle$.
	
	Second, we may assume that $G$ is simply connected, because a $Q$-equivariant isogeny 
$f\colon G\to H$ restricts to an isogeny $f^{-1}(H^Q)\to H^Q$ and $G^Q$ is a subgroup of finite index 
in $f^{-1}(H^Q)$, by the same argument as in the first paragraph. 

Now the result follows from \cite[Theorem 8.1]{Steinberg}.
\end{proof}

\begin{corollary}\label{cor:reductive_fine}
	Assume that $\AAut(\cA)$ is reductive and let $\Gamma$ be a fine abelian group grading on $\cA$. 
If $\chr\FF=0$, then $\Stab(\Gamma)=\Diag(\Gamma)$. If $\chr\FF=p$, then the index 
$[\Stab(\Gamma):\Diag(\Gamma)]$ is a power of $p$.
\end{corollary}

\begin{proof}
	Let $Q=\Diag(\Gamma)$. Then, by Corollary~\ref{cor:eigen}, $\Gamma$ is the eigenspace 
decomposition of $\cA$ with respect to $Q$, so $\Stab(\Gamma)$ is the centralizer of $Q$ in 
$\Aut(\cA)$, which is reductive by Lemma~\ref{lm:red_cent}. By Lemma~\ref{re:stab0}, we then get 
$\Stab(\Gamma)^\circ=Q^\circ$. Hence, for any $s\in\Stab(\Gamma)$, we have $s^n\in Q$ for 
some $n>0$. If $\chr\FF=0$, it follows that $s$ is semisimple and, therefore, $\langle Q,s\rangle$ 
is diagonalizable, which forces $s\in Q$ by maximality of $Q$. If $\chr\FF=p$, choose $n$ 
minimal possible and write $n=mp^k$ where $p\nmid m$. Applying the above argument to $s^{p^k}$, 
we see that $s^{p^k}\in Q$. 
\end{proof}

\begin{proposition}\label{prop:reductive}
	Assume that $\AAut(\cA)$ is reductive. Then, for an abelian group grading $\Gamma$ on $\cA$, 
the following conditions are equivalent:
	\begin{enumerate}
		\item[(i)] $\Gamma$ is almost fine;
		\item[(ii)] $\Diag(\Gamma)^\circ=\Stab(\Gamma)^\circ$;
		\item[(iii)] $\rank(U_\ab(\Gamma))=\dim\cD_e$ where $\cD_e\subset\Der(\cA)$ is defined 
by \eqref{eq:def_L_e}. 
	\end{enumerate}
	If these conditions hold, $\dim\cD_e=\trank(\Gamma)$ and the elements of $\cD_e$ act as scalars
on each component of $\Gamma$.
\end{proposition}

\begin{proof}
	We will see later (Corollary~\ref{cor:no_p_torsion}) that if $\Gamma$ is almost fine, 
then $U_\ab(\Gamma)$ has no $p$-torsion in the case $\chr\FF=p$. Then the  argument in the proof 
of Corollary~\ref{cor:reductive_fine} shows that (i) $\Rightarrow$ (ii). The converse is trivial.
	
	We always have $\rank(U_\ab(\Gamma))=\dim\Diag(\Gamma)$ and 
$\dim\Stab(\Gamma)\le\dim\cD_e$, since $\cD_e$ is the Lie algebra of $\Stabs(\Gamma)$. But being reductive, 
$\AAut(\cA)$ is in particular smooth, so $\Stabs(\Gamma)$, as  the centralizer
of a diagonalizable group scheme, is smooth, too, and this means 
$\dim\Stab(\Gamma)=\dim\cD_e$. It is now clear that (ii) $\Leftrightarrow$ (iii).
\end{proof}

In particular, if $\chr\FF=0$ and $\Gamma$ is an abelian group grading on a semisimple Lie algebra 
$\cL$, then $\ad\colon \cL\to\cD=\Der(\cL)$ is an isomorphism of graded algebras (for any realization of 
$\Gamma$ over an abelian group), so $\cL_e\simeq\cD_e$. Thus, $\Gamma$ is almost fine if and only 
if the quasitorus $Q=\Diag(\Gamma)$ satisfies $\dim Q=\dim\cL_e$, which is \emph{condition (*)} of 
Jun Yu \cite{YuS,YuE}, who studied such quasitori in the automorphism groups of simple complex 
Lie algebras. At the extreme values of toral rank for these almost fine gradings, we have the 
Cartan grading for which $\cL_e$ is a Cartan subalgebra of $\cL$ and \emph{special gradings} of 
Wim Hesselink \cite{Hesselink} for which $\cL_e=0$.

We note that, in general, if a grading $\Gamma$ on $\cA$ satisfies $\cD_e=0$ then $\Gamma$ is 
almost fine of toral rank $0$.

\begin{example}\label{ex:B2}
	Let $\HH$ be the split quaternion algebra over $\FF$, $\chr\FF\ne 2$, with basis 
$\{\hat{1},\hat{\imath},\hat{\jmath},\hat{k}\}$ and multiplication defined by 
$\hat{\imath}^2=\hat{\jmath}^2=1$ and 
$\hat{\imath}\hat{\jmath}=-\hat{\jmath}\hat{\imath}=\hat{k}$. We have a grading on $\HH$ by 
the Klein group $\ZZ_2^2$: $\deg\hat{1}=(\zero,\zero)$, $\deg\hat{\imath}=(\one,\zero)$, 
$\deg\hat{\jmath}=(\zero,\one)$, and $\deg\hat{k}=(\one,\one)$, so we can define a 
$\ZZ_2^3$-grading on $M_2(\HH)\simeq M_2(\FF)\ot\HH$ by setting 
$\deg(E_{ij}\ot d)=(\bar{i}-\bar{j},\deg d)\in\ZZ_2\times\ZZ_2^2$ for any nonzero 
homogeneous $d\in\HH$. 
	
	Denote by bar the standard involution of $\HH$, which maps $\hat{1}\mapsto\hat{1}$, 
$\hat{\imath}\mapsto-\hat{\imath}$,  $\hat{\jmath}\mapsto-\hat{\jmath}$, 
$\hat{k}\mapsto-\hat{k}$. The corresponding involution $*$ on $M_2(\HH)$, 
$E_{ij}\ot d\mapsto E_{ji}\ot\bar{d}$, preserves degrees, so the Lie algebra of skew elements 
	\[
	\cL=\{X\in M_2(\HH)\mid X^*=-X\}
	\]
	becomes $\ZZ_2^3$-graded. This is a simple Lie algebra of type $B_2$, which is isomorphic to 
the algebra of derivations of either itself or the associative algebra with involution $M_2(\HH)$. Since 
$\cL_e=0$, we have almost fine gradings of toral rank $0$ on $\cL$ and  $M_2(\HH)$. However, 
these gradings are not fine, because they can be refined to $\ZZ_2^4$-gradings. Indeed, transporting 
the $\ZZ_2^2$-grading via the isomorphism  $\HH\to M_2(\FF)$ defined by
	\[
	\hat{\imath}\mapsto\matr{1&0\\0&-1},\;\hat{\jmath}\mapsto\matr{0&1\\1&0},
	\]
	we obtain a $\ZZ_2^2$-grading on $M_2(\FF)$, which is a refinement of the original 
$\ZZ_2$-grading $\deg E_{ij}=\bar{i}-\bar{j}$. Consequently,  we obtain a 
$\ZZ_2^2\times\ZZ_2^2$-grading on $M_2(\HH)\simeq M_2(\FF)\ot\HH$, which is a refinement of 
the original $\ZZ_2\times\ZZ_2^2$-grading and is still preserved by the involution $*$ 
(cf. \cite[Theorem~3.30 and Remark~6.60]{EKmon}).
\end{example}

\subsection{Canonical almost fine refinement}

Given a grading $\Gamma$ on $\cA$, pick a maximal torus $T$ in $\Stab(\Gamma)$. Then the 
eigenspace decomposition of each homogeneous component $\cA_s$ with respect to the action of 
$T$ yields a refinement of $\Gamma$:
\begin{equation}\label{eq:def_GT}
	\cA=\kern -5pt\bigoplus_{(s,\lambda)\in S\times\frX(T)}\kern -10pt\cA_{(s,\lambda)}\quad
	\text{ with } \cA_{(s,\lambda)}\bydef\{a\in\cA_s\mid\tau(a)=\lambda(\tau)a\;\forall\tau\in T\},
\end{equation}
where $\frX(T)$ denotes the group of characters of $T$, i.e., the algebraic group homomorphisms from 
$T$ to the multiplicative group $\FF^\times$. We will denote this refinement by $\Gamma^*_T$. 
Clearly, if $\Gamma$ is a $G$-grading for some group $G$, then $\Gamma^*_T$ is a 
$G\times\frX(T)$-grading.

\begin{lemma}\label{lm:indep_of_torus}
	If $T$ and $T'$ are maximal tori of $\Stab(\Gamma)$, then the gradings $\Gamma^*_T$ and 
$\Gamma^*_{T'}$ are equivalent.
\end{lemma}

\begin{proof}
	Since $\FF$ is assumed to be algebraically closed, there exists $\vphi\in\Stab(\Gamma)$ such that 
$\vphi T\vphi^{-1}=T'$ or, in other words, $T'=(\Int\vphi)(T)$, where $\Int\vphi$ is the 
inner automorphism determined by $\vphi$. Thus we get an isomorphism 
$\hat{\vphi}\colon \frX(T')\to\frX(T)$ sending $\lambda'\mapsto\lambda'\circ\Int\vphi$, and it follows from
 the definition that $\vphi\big(\cA_{(s,\lambda)}\big)=\cA'_{(s,\hat{\vphi}^{-1}(\lambda))}$ for all 
$s\in S$, $\lambda\in\frX(T)$.
\end{proof}

\begin{lemma}\label{lm:GT}
	$\Diag(\Gamma^*_T)^\circ=T$ is a maximal torus in $\Stab(\Gamma^*_T)$. In particular, 
$\Gamma^*_T$ is an almost fine grading and $\trank(\Gamma^*_T)=\trank(\Gamma)$.
\end{lemma}

\begin{proof}
	By definition, every element $\tau\in T$ acts as the scalar $\lambda(\tau)$ on 
$\cA_{(s,\lambda)}$, so we have $T\subset\Diag(\Gamma^*_T)^\circ\subset\Stab(\Gamma^*_T)$. 
Since $\Gamma^*_T$ is a refinement of $\Gamma$, we also have 
$\Stab(\Gamma^*_T)\subset\Stab(\Gamma)$. By the maximality of the torus $T$ in $\Stab(\Gamma)$, 
we conclude that $T=\Diag(\Gamma^*_T)^\circ$ is a maximal torus in $\Stab(\Gamma^*_T)$.
\end{proof}

\begin{corollary}\label{co:trank_max}
	For any grading $\Gamma$, $\trank(\Gamma)$ is the maximum of $\dim\Diag(\Gamma')$'s over 
all refinements $\Gamma'$ of $\Gamma$. If $\Gamma$ is a group (respectively, abelian group) 
grading, then this maximum is attained among group (respectively, abelian group) gradings.
\end{corollary}

\begin{proof}
	If $\Gamma'$ is a refinement of $\Gamma$, then 
$\Diag(\Gamma')\subset\Stab(\Gamma')\subset\Stab(\Gamma)$. Since $\Diag(\Gamma')^\circ$ is a 
torus, we get $\dim\Diag(\Gamma')\le\trank(\Gamma)$. The result now follows by Lemma~\ref{lm:GT}.
\end{proof}

The last two lemmas justify the following terminology:

\begin{df}
	For any maximal torus $T\subset\Stab(\Gamma)$, the refinement $\Gamma^*_T$ will be called 
the \emph{canonical almost fine refinement of $\Gamma$}.
\end{df}

The following is an abstract characterization of $\Gamma^*_T$ among refinements of $\Gamma$.

\begin{proposition}\label{prop:characterize_can_refinement}
	The following are equivalent for a refinement $\Gamma'$ of $\Gamma$:
	\begin{enumerate}
		\item[(i)] $\Gamma'$ is almost fine and $\trank(\Gamma')=\trank(\Gamma)$;
		\item[(ii)] $\Gamma'$ is a refinement of $\Gamma^*_T$ for some $T$.
	\end{enumerate}
\end{proposition}

\begin{proof}
	If $\Gamma'$ is a refinement of $\Gamma^*_T$, then $\Gamma'$ satisfies (i) by 
Lemma~\ref{lm:GT} and Proposition~\ref{re:refinement}. Conversely, if $\Gamma'$ satisfies (i), 
then $T\bydef\Diag(\Gamma')^\circ$ is a maximal torus in $\Stab(\Gamma')$ and, hence, in 
$\Stab(\Gamma)$, since $\Stab(\Gamma)$ and its subgroup $\Stab(\Gamma')$ have the same rank 
by hypothesis. But the elements of $T$ act as scalars on each component of $\Gamma'$, so
 $\Gamma'$ must be a refinement of $\Gamma^*_T$.
\end{proof}

\section{Classification of group gradings up to isomorphism}\label{se:classification_iso}

We will now show how a classification of almost fine group gradings on $\cA$ up to equivalence can 
be used to obtain, for any group $G$, a classification of $G$-gradings on $\cA$ up to isomorphism. 

Let $\Gamma$ be a $G$-grading on $\cA$. As discussed in the introduction, $\Gamma$ can be 
obtained from a fine group grading $\Delta$ by a homomorphism $\alpha\colon U(\Delta)\to G$, but neither 
$\Delta$ nor $\alpha$ is unique. To remedy the situation, we restrict the class of homomorphisms 
$\alpha$ that we are going to use, and this forces us to extend the class of gradings from which we 
will take $\Delta$ by allowing $\Delta$ to be almost fine.

\begin{df}\label{df:admissible}
	Let $\Delta$ be an almost fine group grading on $\cA$, $U=U(\Delta)$, $U_\ab=U_\ab(\Delta)$, 
and let $\pi_\Delta\colon U\to U_\ab/t(U_\ab)$ be the composition of the natural homomorphisms 
$U\to U_\ab\to U_\ab/t(U_\ab)$ . A group homomorphism $\alpha\colon U\to G$ is said to be \emph{admissible} if the restriction of the homomorphism 
$(\alpha,\pi_\Delta)\colon U\to G\times U_\ab/t(U_\ab)$ to the support of $\Delta$ is injective.
\end{df}

In the abelian case, i.e., if $G$ is an abelian group and $\Delta$ is an abelian group grading, 
the restriction of the natural homomorphism  $\pi_\ab\colon U\to U_\ab$ to the support of $\Delta$ is 
injective and any homomorphism $\alpha\colon U\to G$ is the composition of  $\pi_\ab$ and a 
(unique) homomorphism $\alpha'\colon U_\ab\to G$. 
Hence, the condition in Definition~\ref{df:admissible} reduces to the following: the restriction of 
$(\alpha',\pi'_\Delta)$ to the support of $\Delta$ is injective, where $\pi'_\Delta$ is the 
natural homomorphism $U_\ab\to U_\ab/t(U_\ab)$. We will say that $\alpha'$ is \emph{admissible} 
if this is satisfied.

\begin{lemma}\label{lm:admissible}
	Let $\Delta$ be an almost fine group grading (respectively, abelian group grading) on $\cA$ and 
let $G$ be a group (respectively, abelian group).
	Denote $T=\Diag(\Delta)^\circ$. Let $\Gamma={}^\alpha\Delta$ be the $G$-grading induced by 
a homomorphism $\alpha\colon U(\Delta)\to G$ (respectively, $\alpha\colon U_\ab(\Delta)\to G$) . Then the 
following are equivalent: 
	\begin{enumerate}
		\item[(i)] $\alpha$ is admissible;
		\item[(ii)] $T$ is a maximal torus in $\Stab(\Gamma)$ and the set of nonzero
 homogeneous components of $\Gamma^*_T$ coincides with that of $\Delta$ (in particular, these 
gradings are equivalent).
	\end{enumerate}
\end{lemma}

\begin{proof}
	Recall from Subsection \ref{sse:almost_fine} that the torus $T=\Diag(\Delta)^\circ$ is isomorphic 
to the group of characters of $U_\ab/t(U_\ab)$ where $U_\ab=U_\ab(\Delta)$. Hence, we obtain 
an evaluation homomorphism $\veps\colon U=U(\Delta)\to\frX(T)$, which is  the composition of 
$\pi_\Delta$ and the isomorphism $U_\ab/t(U_\ab)\to\frX(T)$. Explicitly, for any $s$ in the support of 
$\Delta$, $\veps(s)$ is the character of $T$ that maps each $\tau\in T$ to the scalar by which $\tau$ 
acts on the component $\cA_s$ of $\Delta$. It follows that the induced $G\times\frX(T)$-grading 
${}^{(\alpha,\veps)}\Delta$ coincides with the $G\times\frX(T)$-grading $\Gamma'$ obtained from 
$\Gamma$ by decomposing each of its components into eigenspaces with respect to the action 
of $T\subset\Stab(\Gamma)$.
	
	Now, if (ii) holds, then $\Gamma'=\Gamma^*_T$ by definition and, hence, the coarsening 
${}^{(\alpha,\veps)}\Delta$ of $\Delta$ is not proper, so the restriction of $(\alpha,\veps)$ to 
the support of $\Delta$ is injective. Since $\veps$ is the composition of $\pi_\Delta$ and an 
isomorphism, we get (i).
	
	Conversely, assume (i). Then the restriction of $(\alpha,\veps)$ to the support of $\Delta$ is 
injective, so $\Gamma'={}^{(\alpha,\veps)}\Delta$ has the same nonzero homogeneous components as 
$\Delta$ and, hence, $\Stab(\Gamma')=\Stab(\Delta)$. But 
$\Stab(\Gamma')=\Cent_{\Stab(\Gamma)}(T)$, so $T$ is a maximal torus in 
$\Cent_{\Stab(\Gamma)}(T)$, since $\Delta$ is almost fine. It follows that $T$ is a maximal torus in 
$\Stab(\Gamma)$. Since $\Gamma'=\Gamma^*_T$, we see that (ii) holds.
\end{proof}

\begin{theorem}\label{th:1}
Let $\{\Gamma_i\}_{i\in I}$ be a set of representatives of the equivalence classes of almost fine 
group (respectively, abelian group) gradings on $\cA$. For any group (respectively, abelian group) $G$ 
and a $G$-grading $\Gamma$ on $\cA$, there exists a unique $i\in I$ such that $\Gamma$ is 
isomorphic to the induced grading ${}^\alpha\Gamma_i$ for some admissible homomorphism 
$\alpha\colon U(\Gamma_i)\to G$ (respectively, $\alpha\colon U_\ab(\Gamma_i)\to G$). Moreover, two 
such homomorphisms, $\alpha$ and $\alpha'$, induce isomorphic $G$-gradings if and only if there 
exists $w\in W(\Gamma_i)$ such that $\alpha=\alpha'\circ w$.
\end{theorem}

\begin{proof}
Consider the canonical almost fine refinement $\Gamma^*=\Gamma^*_T$, for some maximal 
torus $T\subset\Stab(\Gamma)$. Since $\Gamma^*$ is equivalent to some $\Gamma_i$, there exists 
an automorphism $\vphi$ of $\cA$ that moves the set of nonzero homogeneous components of 
$\Gamma^*$ onto that of $\Gamma_i$. Hence, $\varphi(\Gamma)$ is a coarsening of $\Gamma_i$, 
so there exists a homomorphism $\alpha\colon U(\Gamma_i)\to G$ such that 
$\vphi(\Gamma)={}^\alpha\Gamma_i$. Since $T=\Diag(\Gamma^*)^\circ$ by Lemma~\ref{lm:GT}, 
we have $\vphi T\vphi^{-1}=\Diag(\Gamma_i)^\circ$ and can apply Lemma~\ref{lm:admissible}, with 
$\Delta=\Gamma_i$, to conclude that $\alpha$ is admissible.

Now, suppose that $\alpha\colon U(\Gamma_i)\to G$ and $\alpha'\colon U(\Gamma_j)\to G$ are 
admissible homomorphisms such that the induced $G$-gradings ${}^{\alpha}\Gamma_i$ and 
${}^{\alpha'}\Gamma_j$ are isomorphic, i.e., there exists $\vphi\in\Aut(\cA)$ such that 
$\vphi({}^{\alpha}\Gamma_i)={}^{\alpha'}\Gamma_j$. In particular, we have  
$\vphi\Stab({}^{\alpha}\Gamma_i)\vphi^{-1}=\Stab({}^{\alpha'}\Gamma_j)$. 
Let $T=\Diag(\Gamma_i)^\circ$ and $T'=\vphi^{-1}\Diag(\Gamma_j)^\circ\vphi$. 
Applying Lemma~\ref{lm:admissible} to $\Delta=\Gamma_i$ and to $\Delta=\Gamma_j$, we see that 
$T$ and $T'$ are maximal tori of $\Stab({}^{\alpha}\Gamma_i)$, $\Gamma_i$ is equivalent to 
$({}^{\alpha}\Gamma_i)^*_T$, and $\Gamma_j$ is equivalent to 
$({}^{\alpha}\Gamma_i)^*_{T'}$. But $({}^{\alpha}\Gamma_i)^*_T$ and 
$({}^{\alpha}\Gamma_i)^*_{T'}$ are equivalent by Lemma~\ref{lm:indep_of_torus}, so 
$\Gamma_i$ and $\Gamma_j$ are equivalent, which forces $i=j$.

Finally, since $T$ and $T'=\vphi^{-1} T\vphi$ are maximal tori of $\Stab({}^{\alpha}\Gamma_i)$, 
there exists $\psi\in\Stab({}^{\alpha}\Gamma_i)$ such that $T'=\psi T\psi^{-1}$. Replacing $\vphi$ 
by the composition $\vphi\psi$, we get  $T=\vphi^{-1} T\vphi$ and still 
$\vphi({}^{\alpha}\Gamma_i)={}^{\alpha'}\Gamma_i$. But then $\vphi$ moves the set of 
nonzero homogeneous components of $({}^{\alpha}\Gamma_i)^*_T$ onto that of 
$({}^{\alpha'}\Gamma_i)^*_T$. Since, by Lemma~\ref{lm:admissible}, these sets are both equal to 
the set of nonzero homogeneous components of $\Gamma_i$, we see that $\vphi\in\Aut(\Gamma_i)$ 
and, hence, $\vphi$ determines an element $w\in W(\Gamma_i)$ by $\vphi(\cA_s)=\cA_{w(s)}$ for all 
$s$ in the support $S$ of $\Gamma_i$. Now, the homogeneous component of degree $g\in G$ in 
the grading ${}^{\alpha}\Gamma_i$ is, by definition, the direct sum 
$\bigoplus_{s\in\alpha^{-1}(g)}\cA_s$, whereas in ${}^{\alpha'}\Gamma_i$ it is 
$\bigoplus_{s'\in\alpha'^{-1}(g)}\cA_{s'}$. Since $\vphi$ moves the former to the latter, we 
conclude that, for any $s\in S$, $\vphi(\cA_s)\subset\bigoplus_{s'\in\alpha'^{-1}(\alpha(s))}\cA_{s'}$ 
and, hence, $w(s)\in\alpha'^{-1}(\alpha(s))$. This implies $\alpha'(w(s))=\alpha(s)$ for all $s\in S$, 
so $\alpha=\alpha'\circ w$.
\end{proof}

\begin{remark}\label{re:af_weyl}
Theorem \ref{th:1} reduces the problem of classification of $G$-gradings up to isomorphism to 
the problems
of classifying almost fine gradings up to equivalence and of describing their Weyl groups, as subgroups
of automorphisms of their universal groups.
Weyl groups are important invariants reflecting the symmetries of gradings, so their computation 
is of independent interest
and usually far from trivial (see \cite{EKmon} and references therein).
\end{remark}

\section{From fine to almost fine gradings}\label{se:fine_to_almost_fine}

As we have seen in the previous section, the knowledge of almost fine group gradings on $\cA$ up 
to equivalence, together with their universal and Weyl groups, yields a classification of all $G$-gradings 
on $\cA$ up to isomorphism, for any group $G$. We will now discuss, in the abelian case, how
 to determine almost fine gradings if fine gradings are known, which can then be used to classify all 
$G$-gradings for abelian $G$.

Recall from Section~\ref{se:preliminaries} that, if $\AAut(\cA)$ is smooth (as is always the case
 in characteristic $0$), then fine abelian group gradings on $\cA$ are classified by 
the conjugacy classes of maximal diagonalizable subgroups of $\Aut(\cA)$, which can be studied using 
the tools of the theory of algebraic groups or, in characteristic $0$,  of compact Lie groups, since in 
that case the problem reduces to the field of complex numbers (see \cite{Eld16}).  For example, 
fine gradings on exceptional simple Lie algebras and superalgebras over an algebraically closed field 
of characteristic $0$ were classified in this way (see \cite{DEM,EKmon,YuE} and the references therein).
Also note that, for a simple Lie (super)algebra, the universal group of any grading is abelian
 (see, e.g., \cite[Proposition 1.12]{EKmon}).

\begin{proposition}\label{prop:2}
Let $\Delta$ be a fine group (respectively, abelian group) grading on $\cA$ and let $\Gamma$ be 
a coarsening of $\Delta$. Then $\Gamma$ is almost fine if and only if the kernel of the quotient 
map $U_\ab(\Delta)\to U_\ab(\Gamma)$ is finite and $\trank(\Gamma)=\trank(\Delta)$.
\end{proposition}

\begin{proof}
Denote $U=U_\ab(\Delta)$ and let $E\subset U$ be the above kernel. 
Since the canonical almost fine refinement of $\Delta$ cannot be proper, $\Delta$ is almost fine, so 
we have $\rank(U)=\trank(\Delta)$. Now, tensoring the short exact sequence of abelian groups 
$0\to E\to U\to U/E\to 0$ by $\QQ$ over $\ZZ$ 
yields $\rank(U/E)=\rank(U)-\rank(E)$. Since $\Gamma$ is a coarsening of $\Delta$, we also have 
$\trank(\Delta)\le\trank(\Gamma)$. Therefore, $\rank(U/E)=\trank(\Gamma)$ if and only if 
$\rank(E)=0$ and $\trank(\Delta)=\trank(\Gamma)$.
\end{proof}

Now let $\Gamma$ be an almost fine abelian group grading. Then $\Gamma$ is a coarsening of some 
fine abelian group grading $\Delta$, hence $\Gamma$ is defined by the quotient map 
$U_\ab(\Delta)\to U_\ab(\Gamma)$, whose kernel must be finite by Proposition~\ref{prop:2}. 
It follows that Proposition~\ref{prop:no_p_torsion} and Corollary~\ref{cor:eigen} extend to 
almost fine gradings:

\begin{corollary}\label{cor:no_p_torsion}
Assume $\AAut(\cA)$ is smooth. If $\Gamma$ is an almost fine abelian group grading on $\cA$, 
then $U_\ab(\Gamma)$ has no $p$-torsion in the case $\chr\FF=p$ and, hence, $\Gamma$ is 
the eigenspace decomposition 
with respect to $\Diag(\Gamma)$ in any characteristic.
\end{corollary}

\begin{corollary}
Any fine abelian group grading admits only finitely many almost fine coarsenings that are 
themselves abelian group gradings.
\end{corollary}

Enumerating almost fine coarsenings is helped by the fact that the subgroups of $U_\ab(\Delta)$ 
that lie in the same $W(\Delta)$-orbit correspond to equivalent coarsenings.

\begin{theorem}\label{th:3}
Assume $\AAut(\cA)$ is reductive. Let $\Delta$ be a fine abelian group grading on 
$\cA$, $U=U_\ab(\Delta)$, and $\Sigma$ 
be the support of the induced $U$-grading on the Lie algebra $\cD=\Der(\cA)$. Let $\Gamma$ be 
an abelian group grading that is a coarsening of $\Delta$ and let $E$ be the kernel of the quotient 
map $U\to U_\ab(\Gamma)$.
Then $\Gamma$ is almost fine if and only if $E\subset t(U)$ and $E\cap\Sigma\subset\{e\}$. 
\end{theorem}

\begin{proof}
Under the additional assumption, we have $\rank(U)=\dim\cD_e$ by 
Proposition~\ref{prop:reductive}. With respect to the $U/E$-grading on $\cD$ induced by 
$\Gamma$, the identity component is $\cD_E\bydef\bigoplus_{g\in E}\cD_g$. Hence, $\Gamma$ 
is almost fine if and only if $\rank(U/E)=\dim\cD_E$ (again by Proposition~\ref{prop:reductive}). 
Then we proceed as in the proof of  Proposition~\ref{prop:2}, but the condition 
$\trank(\Gamma)=\trank(\Delta)$ is replaced with $\dim\cD_E=\dim\cD_e$, which is equivalent 
to $E\cap\Sigma\subset\{e\}$. 
\end{proof}

\begin{example}\label{ex:A3}
If $\chr\FF\ne 2$, the Lie algebra $\cL=\Sl_4(\FF)$ is simple of type $A_3$ and has a fine 
$\ZZ_2^4$-grading $\Delta$ obtained by refining, by means of the outer automorphism 
$X\mapsto-X^T$, the $\ZZ_2^3$-grading induced from the Cartan grading by the ``mod $2$'' map 
$\ZZ^3\to\ZZ_2^3$ (see e.g. \cite[Example 3.60]{EKmon}). Explicitly, $\Delta$ is the restriction to 
$\Sl_4(\FF)$ of the $\ZZ_2\times(\ZZ_2^4)_0$-grading on $\Gl_4(\FF)$ defined by setting 
$\deg(E_{ij}-E_{ji})=(\zero,\veps_i-\veps_j)$ and $\deg(E_{ij}+E_{ji})=(\one,\veps_i-\veps_j)$ where 
$\{\veps_1,\veps_2,\veps_3,\veps_4\}$ is the standard basis of $\ZZ_2^4$ and 
$(\ZZ_2^4)_0\simeq\ZZ_2^3$ is the span of $\veps_i-\veps_j$. We have $\cL\simeq\Der(\cL)$, 
$\cL_e=0$ and, moreover, $\cL_g=0$ for $g\in\ZZ_2\times\{(\one,\one,\one,\one)\}$. It follows 
that, for each of the two possible values of $g$, the $\ZZ_2^3$-grading on $\cL$ induced by the 
natural homomorphism $\ZZ_2^4\to\ZZ_2^4/\langle g\rangle$ is almost fine. In fact, these two 
almost fine gradings are equivalent, because the two values of $g$ are in the same 
$W(\Delta)$-orbit (see \cite[Example 3.63]{EKmon}).
\end{example}

\section{Root systems associated to non-special gradings on semisimple Lie algebras}\label{se:Lss}

In this section $\cL$ will be a semisimple finite-dimensional Lie algebra over an algebraically closed
field $\FF$ of characteristic $0$. The aim is to show that a (possibly nonreduced) root system of rank $r$ 
can be attached canonically to any abelian group grading $\Gamma$ on $\cL$ of toral rank $r\neq 0$, 
i.e., to any non-special $\Gamma$. We will take advantage of the results in \cite{Eld_root} that deal 
with the case when $\Gamma$ is fine. 
For the definition of possibly nonreduced root systems, see e.g. \cite[Ch.~VI,\S 1]{BouVI}.

Let $G$ be an abelian group and let $\Gamma:\cL=\bigoplus_{g\in G}\cL_g$ be a $G$-grading on $\cL$ 
with $\trank(\Gamma)\geq 1$. Let $T$ be a maximal torus in $\Stab(\Gamma)$. It induces a weight
space decomposition:
\begin{equation}\label{eq:Lss_weights}
\cL=\bigoplus_{\lambda\in \frX(T)}\cL(\lambda)
\end{equation}
where $\cL(\lambda)=\{x\in\cL\mid \tau(x)=\lambda(\tau)x\;\forall \tau\in T\}$. 

Let $\Gamma_T^*$ be the associated almost fine grading, as in \eqref{eq:def_GT}:
\begin{equation}\label{eq:Lss_GT}
\Gamma_T^*:\cL
=\kern -10pt\bigoplus_{(g,\lambda)\in G\times\frX(T)} \kern -10pt\cL_g\cap \cL(\lambda).
\end{equation}
Let $\cH$ be the Lie algebra of $T$ inside $\cL\simeq\Der(\cL)$, so
$\cH$ is a Cartan subalgebra of the reductive Lie subalgebra $\cL_e$. 
The adjoint action of $\cH$ on any weight space $\cL(\lambda)$ is given by the differential 
$\textup{d}\lambda\in\cH^*$, 
which is therefore a weight of the adjoint action of $\cH$ on $\cL$. To simplify notation, we will use 
$\lambda$ to denote its differential, too, and thus identify $\frX(T)$ with a subgroup of $\cH^*$. 

Denote by $\Phi$ the set of nonzero weights of $\cH$ on $\cL$: 
\[
\Phi=\{\lambda\in\cH^*\smallsetminus\{0\} \mid \cL(\lambda)\neq 0\}.
\] 
Under the above identification, we have $\ZZ\Phi=\frX(T)$ and also 
$\cH\subset\cL_{(e,0)}=\cL_e\cap\cL(0)$, where $0$ is (the differential of) the trivial character on $T$.
Since $\Gamma_T^*$ is almost fine, we have equality: $\cH=\cL_e\cap\cL(0)$ 
by Proposition~\ref{prop:reductive}, as the connected component of $\Aut(\cL)$ is semisimple and, hence,
$\AAut(\cL)$ is reductive.

\begin{theorem}\label{th:Lss_rootsystem}
With the hypotheses above, $\Phi$ is a (possibly nonreduced) root system in $\RR\otimes_\ZZ\ZZ\Phi$. 
If $\cL$ is simple, then $\Phi$ is an irreducible root system.
\end{theorem}
\begin{proof}
 Let $\Gamma'$ be a 
refinement of $\Gamma_T^*$ that is a fine abelian group grading, and let $U$ be its universal 
abelian group. 
Lemma \ref{lm:GT} and Proposition \ref{re:refinement}
show that $T=\Diag(\Gamma_T^*)^\circ=\Diag(\Gamma')^\circ$ is a maximal torus in $\Stab(\Gamma')$, in particular 
$U/t(U)\simeq\frX(T)$. Now \cite[Theorem~4.4]{Eld_root} (or \cite[Theorem~6.61]{EKmon}) 
gives the result.
\end{proof}

Abelian group gradings on semisimple Lie algebras have been reduced to gradings on simple Lie
algebras in \cite{CE}. For simple Lie algebras, Theorem \ref{th:Lss_rootsystem} implies that any 
non-special abelian group grading is related to a grading by a root system. These gradings were first
studied by S.~Berman and R.V.~Moody \cite{BermanMoody}.

\begin{df}\label{df:reduced_root_graded}
A Lie algebra $\cL$ over $\FF$ is \emph{graded by the reduced root system $\Phi$}, or  
\emph{$\Phi$-graded}, if the following conditions are satisfied:
\begin{romanenumerate}
\item $\cL$ contains as a subalgebra a finite-dimensional simple Lie algebra 
whose root system relative to a Cartan subalgebra $\frh=\frg_0$ is $\Phi$:  
$\frg=\frh\oplus\bigl(\bigoplus_{\alpha\in\Phi}\frg_\alpha\bigr)$;
\item $\cL=\bigoplus_{\alpha\in\Phi\cup\{0\}}\cL(\alpha)$, where $\cL(\alpha)
=\{x\in\cL\mid [h,x]=\alpha(h)x\ \text{for all}\ h\in\frh\}$; 
\item $\cL(0)=\sum_{\alpha\in\Phi}[\cL(\alpha),\cL(-\alpha)]$.
\end{romanenumerate}
The subalgebra $\frg$ is said to be a \emph{grading subalgebra} of $\cL$.
\end{df}

The simply laced case (i.e., types $A_r$, $D_r$ and $E_r$) was studied in \cite{BermanMoody}, 
and G.~Benkart and E.~Zelmanov  considered the remaining  cases in \cite{BZ}.

As to nonreduced root systems, the definition works as follows (see \cite{ABG}):

\begin{df}\label{df:nonreduced_root_graded}
Let $\Phi$ be the nonreduced root system $BC_r$ ($r\geq 1$). A Lie algebra $\cL$ over $\FF$ 
is \emph{graded by $\Phi$}, or \emph{$\Phi$-graded}, if the following conditions are satisfied:
\begin{romanenumerate}
\item $\cL$ contains as a subalgebra a finite-dimensional simple Lie algebra 
$\frg=\frh\oplus\bigl(\bigoplus_{\alpha\in\Phi'}\frg_\alpha\bigr)$ whose root system $\Phi'$ relative to 
a Cartan subalgebra $\frh=\frg_0$ is the reduced subsystem of type $B_r$, $C_r$ or $D_r$ contained in 
$\Phi$;
\item $\cL=\bigoplus_{\alpha\in\Phi\cup\{0\}}\cL(\alpha)$, where 
$\cL(\alpha)=\{x\in\cL\mid [h,x]=\alpha(h)x\,\ \text{for all}\ h\in\frh\}$;
\item $\cL(0)=\sum_{\alpha\in\Phi}[\cL(\alpha),\cL(-\alpha)]$.
\end{romanenumerate}
Again, the subalgebra $\frg$ is said to be a \emph{grading subalgebra} of $\cL$, and $\cL$ is said to 
be \emph{$BC_r$-graded with grading subalgebra of type $X_r$}, where $X_r$ is the type of $\frg$.
\end{df}

Assume from now on that $\cL$ is a finite-dimensional simple Lie algebra over our algebraically 
closed field 
$\FF$ of characteristic $0$ and let $\Gamma:\cL=\bigoplus_{g\in G}\cL_g$ be a non-special grading 
by an abelian group $G$. As in the proof of Theorem \ref{th:Lss_rootsystem}, let $\Gamma_T^*$ be 
the canonical almost fine refinement of $\Gamma$, and refine $\Gamma_T^*$ to a fine abelian 
group grading $\Gamma'=\bigoplus_{u\in U}\cL'_u$, where $U=U_\ab(\Gamma')$ (which coincides 
with $U(\Gamma')$ since $\cL$ is simple \cite[Corollary~1.21]{EKmon}). Then, for any $u$ in the 
support, there
is a unique $\alpha\in\Phi\cup\{0\}$ such that $\cL'_u\subset \cL(\alpha)$, and this induces a 
surjective group homomorphism $\pi\colon U\rightarrow\ZZ\Phi$ with kernel $t(U)$. 

Fix a system $\Delta$ of simple roots of $\Phi$. Then $\Phi=\Phi^+\cup\Phi^-$, with 
$\Phi^+\subset\sum_{\alpha\in\Delta}\ZZ_{\geq 0}\alpha$ and $\Phi^-=-\Phi^+$. 
As in \cite[\S 5]{Eld_root}, we choose, for any $\alpha\in\Delta$, an element $u_\alpha\in U$ such 
that $\pi(u_\alpha)=\alpha$. This gives us a section of the homomorphism $\pi$, so $U$ becomes 
the direct product of the free abelian group $U'$ generated by the elements $u_\alpha$, 
$\alpha\in \Delta$, and its torsion subgroup $t(U)$. For any $\lambda\in\ZZ\Phi$, we will denote 
by $u_\lambda$ the unique element of $U'$ such that $\pi(u_\lambda)=\lambda$.

Now, \cite[Theorem 5.1]{Eld_root} (or \cite[Theorem~6.62]{EKmon}) shows that
\[
\frg\bydef \bigoplus_{u\in U'}\cL'_u 
\]
is a simple Lie algebra with Cartan subalgebra $\frh=\cH$ and a root system $\Phi'\subset\Phi$ such 
that $\Delta$ is a system of simple roots. Moreover, $\cL$ is graded by the irreducible root system 
$\Phi$ with grading subalgebra $\frg$,  and if $\Phi$ is nonreduced (type $BC_r$), then $\frg$ is simple 
of type $B_r$. By its construction, not only $\frg$ but also the components of its triangular 
decomposition $\frg=\frg_-\oplus\frh\oplus\frg_+$ associated to $\Delta$ are graded subalgebras of 
$\cL$ with respect to the fine grading $\Gamma'$ and, hence, also for $\Gamma_T^*$ and $\Gamma$.

In this situation, the adjoint action of $\frg$ on $\cL$ decomposes $\cL$ into a direct sum of 
irreducible submodules of only a few isomorphism classes. Collecting isomorphic submodules, we get 
the corresponding isotypic decomposition (see \cite{ABG}):
\begin{itemize}
\item If $\Phi$ is reduced, then the isotypic decomposition is 
\[
\cL=(\frg\otimes\cA)\oplus(\cW\otimes\cC)\oplus\cD,
\]
where $\cW=0$ if $\Phi$ is simply laced, and otherwise it is the irreducible module whose highest weight, 
relative to $\cH$ and $\Delta$, is the highest short root in $\Phi$. The component $\cD$ is the
sum of trivial one-dimensional modules, so $\cD=\Cent_\cL(\frg)$ is a subalgebra of $\cL$.
Note that $\cA$ contains a distinguished element $1$ that identifies the subalgebra $\frg$ with 
$\frg\otimes 1$.
In this case $\fra\bydef \cA\oplus\cC$ becomes the \emph{coordinate algebra} with identity $1$, 
whose product is determined by the bracket in $\cL$. Depending on the type of $\Phi$, different 
classes of algebras (associative, alternative, Jordan) may appear as coordinate algebras.

\item If $\Phi$ is of type $BC_r$ with grading subalgebra of type $B_r$ and $r\geq 2$, then the isotypic
decomposition is
\[
\cL=(\frg\otimes\cA)\oplus(\frs\otimes\cB)\oplus(\cW\otimes\cC)\oplus\cD,
\]
where $\cW$ is the natural module, of dimension $2r+1$, for the simple Lie algebra 
$\frg\simeq\frso_{2r+1}(\FF)$, 
so $\cW$ is endowed with an invariant symmetric nondegenerate bilinear form $(\cdot\mid\cdot)$. 
Then 
\[
\frs=\{f\in\End_\FF(\cW)\mid (f(v)\mid w)=(v\mid f(w))\;\;\forall u,v\in\cW,\;\tr(f)=0\}.
\]
The subalgebra $\cD$ is again the centralizer of $\frg$, and the \emph{coordinate algebra} is 
$\fra\bydef\cA\oplus\cB\oplus \cC$.

\item If $\Phi$ is of type $BC_1$ with grading subalgebra of type $B_1$, then the adjoint module 
is isomorphic to the natural module and the isotypic decomposition reduces to
\[
\cL=(\frg\otimes\cA)\oplus(\frs\otimes\cB)\oplus\cD,
\]
with \emph{coordinate algebra} $\fra\bydef\cA\oplus\cB$.
\end{itemize}

To simplify notation, we will write $\fra=\cA\oplus\cB\oplus\cC$ in all cases, with the understanding that 
$\cB$ or $\cC$ may be $0$.

It is clear that $\cD=\Cent_{\cL}(\frg)$ is a graded subalgebra with respect to $\Gamma'$. 
If $\lambda$ is the highest root of $\frg$ with respect to $\Delta$, then 
$\frg_\lambda\otimes\cA=\{x\in\cL(\lambda)\mid [\frg_+,x]=0\}$ is a graded subspace of $\cL$ 
with respect to $\Gamma'$. Since $\dim\frg_\lambda=1$, this allows us to define 
a grading on $\cA$ by the torsion subgroup $t(U)$ as follows: $\cA=\bigoplus_{u\in t(U)}\cA'_u$ where 
$\frg_\lambda\otimes\cA'_u=(\frg_\lambda\otimes\cA)\cap\cL'_{u_\lambda u}$. Since  $\frg\otimes\cA$ 
is the 
$\frg$-submodule of $\cL$ generated by $\frg_\lambda\otimes\cA$, it follows that the isotypic 
component $\frg\otimes\cA$ 
is graded and $\frg_\mu\otimes\cA'_u = (\frg\otimes\cA)\cap\cL'_{u_\mu u}$ for all $\mu\in\ZZ\Phi$ 
and $u\in t(U)$. 
The same argument applies to the other possible isotypic components $\frs\otimes\cB$ and 
$\cW\otimes\cC$, 
substituting for $\lambda$ the highest weight of $\frs$ or $\cW$. 
It follows that the coordinate algebra $\fra$ inherits a grading by $t(U)$. 

Now let $\delta\colon U\rightarrow G$ be the group homomorphism obtained from the fact that $\Gamma'$ 
is a refinement of 
$\Gamma$: $\cL_g=\bigoplus_{u\in\delta^{-1}(g)}\cL'_u$ for any $g\in G$. Then $\frg$ and the 
isotypic components are 
graded subspaces of $\cL$ with respect to $\Gamma$, with the $G$-gradings induced by $\delta$. 
We also define a $G$-grading on $\fra$ (and its pieces) via $\delta$. For the (reductive) subalgebra 
$\cL(0)=(\frg_0\otimes\cA)\oplus(\frs_0\otimes\cB)\oplus(\cW_0\otimes\cC)\oplus\cD$, 
the identity component with respect to $\Gamma$ is 
$\cH=\cL(0)_e=(\frg_0\otimes\cA_e)\oplus(\frs_0\otimes\cB_e)\oplus(\cW_0\otimes\cC_e)\oplus\cD_e$.
But $\frg_0=\cH$, so we conclude that $\cA_e=\FF 1$ and $\cB_e=\cC_e=\cD_e=0$. 
On the other hand, since $\cL(0)=\Cent_{\cL}(\cH)$, we have $\cH=\cL(0)_e\subset Z(\cL(0))$. 
Therefore, the restriction of $\Gamma$ to the semisimple Lie algebra $[\cL(0),\cL(0)]$ is a special grading.

The fine grading $\Gamma'$ is not uniquely determined by $\Gamma$. In fact, 
by Proposition~\ref{prop:characterize_can_refinement}, we can take as $\Gamma'$ any fine 
refinement of $\Gamma$ that has the same toral rank. We have obtained the following result:

\begin{theorem}\label{th:Lss_structure}
Let $\cL$ be a finite-dimensional simple Lie algebra over an algebraically closed field $\FF$ 
of characteristic $0$ and let $\Gamma:\cL=\bigoplus_{g\in G}\cL_g$ be a non-special grading on $\cL$ by an abelian group 
$G$. Then there exists a fine refinement $\Gamma'$ of $\Gamma$ whose identity component is a 
Cartan subalgebra $\cH$ of the reductive subalgebra $\cL_e$. Moreover, for any such $\Gamma'$, 
let $U=U_\ab(\Gamma')$ and let $\pi\colon U\to\ZZ\Phi$ and $\delta\colon U\to G$ be homomorphisms defined by 
$\cL'_u\subset\cL(\pi(u))$ and $\cL'_u\subset\cL_{\delta(\alpha)}$, where $\Phi$ is the root system 
as in Theorem~\ref{th:Lss_rootsystem}, associated to the decomposition 
$\cL=\bigoplus_{\alpha\in\Phi\cup\{0\}}\cL(\alpha)$ with respect to the adjoint action of $\cH$. Then 
\begin{enumerate}
\item[(i)] $\pi$ is surjective with kernel $t(U)$.
\item[(ii)]
Any homomorphism $\lambda\mapsto u_\lambda$ splitting $\pi$ defines
a $\Phi$-grading on $\cL$ with the grading subalgebra $\frg:=\bigoplus_\alpha\cL'_{u_\alpha}$ 
(of type $B_r$ if $\Phi$ is $BC_r$) and
a $G$-grading on the coordinate algebra $\fra=\cA\oplus\cB\oplus\cC$, such that the isotypic
components of $\cL$ for the adjoint action of $\frg$ are $G$-graded subspaces, with the grading on
$\frg\otimes\cA$ given by $\deg(\frg_\alpha\otimes\cA_g)=g_\alpha g$ for all $g\in G$, and similarly
for $\frs\otimes\cB$ and $\cW\otimes\cC$ (if applicable), where $g_\alpha=\delta(u_\alpha)$.
\item[(iii)] 
For the $G$-gradings on $\fra$ and $\cL(0)$ as in (ii), the supports are contained in $t(G)$, 
the identity component of $\fra$ is $\FF1$, and the gradings on the subalgebras $\cD$ and 
$[\cL(0),\cL(0)]$ of $\cL(0)$ are special.\qed 
\end{enumerate}
\end{theorem}

\begin{example}
The simple Lie algebra $\cL$ of type $E_8$ is the Lie algebra obtained by means of the Tits construction
using the Cayley algebra $\OO$ and the Albert (i.e., exceptional simple Jordan) algebra $\Alb$: 
$\cL=\Der(\OO)\oplus(\OO_0\otimes\Alb_0)\oplus\Der(\Alb)$
(see e.g. \cite[\S 6.2]{EKmon}). The Cayley algebra is endowed with a $\ZZ_2^3$-grading (a division
grading), and this induces naturally a $\ZZ_2^3$-grading $\Gamma$ on $\cL$. 
The group $\Aut(\Alb)$ (simple of type $F_4$)
embeds naturally in $\Stab(\Gamma)\subset\Aut(\cL)$, and any maximal torus $T$ in $\Aut(\Alb)$ is 
a maximal torus in $\Stab(\Gamma)$. The canonical almost fine refinement $\Gamma_T^*$ is the 
$\ZZ^4\times\ZZ_2^3$-grading obtained by combining the Cartan grading on $\Alb$ (induced by $T$) 
and the 
$\ZZ_2^3$-grading on $\OO$. It happens in this case that $\Gamma_T^*$ is fine. 

Here the root system $\Phi$ is of type $F_4$ and the isotypic decomposition is
given by the components in the Tits construction: $\frg=\Der(\Alb)$, $\cA=\FF 1$, $\cW=\Alb_0$, 
$\cC=\OO_0$, and $\cD=\Der(\OO)$ ($\cB=0$ in this case). The coordinate algebra is 
$\fra=\FF 1\oplus\OO_0$ is just the Cayley algebra $\OO$. The reductive subalgebra $\cL(0)$ is
the direct sum of the Lie algebra of $T$ (a Cartan subalgebra of $\Der(\Alb)$) and the simple Lie
algebra $\Der(\OO)$.
\end{example}

In conclusion, we note that the $\Phi$-grading on $\cL$ defined by a non-special fine grading 
$\Gamma'$ allows us to restate the conditions in Theorem~\ref{th:3} and 
Definition~\ref{df:admissible} quite explicitly. 
Let $S_\fra$ and $S_\cD$ be the supports of the $t(U)$-gradings on $\fra$ and $\cD$, respectively, 
so the support 
of $\cL(0)$ is $S=S_\fra\cup S_\cD$. Then the almost fine coarsenings of $\Gamma'$ are determined 
by the subgroups $E\subset t(U)$ that are generated by some elements of the form $uv^{-1}$ 
with $u,v\in S$ (so $U/E$ is the universal group of the coarsening \cite[Corollary~1.26]{EKmon}) 
and satisfy $E\cap S=\{e\}$. A homomorphism $\gamma\colon U/E\to G$ is admissible if and only if 
its restriction to the support of each $\cL(\alpha)$ is injective, which amounts to $\gamma|_S$ 
being injective.



\begin{thebibliography}{ABG02}

\bibitem[BKR18]{BKR_Lie} Y.~Bahturin, M.~Kochetov, and A.~Rodrigo-Escudero, \emph{Gradings 
on classical central simple real Lie algebras}, J. Algebra \textbf{506} (2018), 1--42.

\bibitem[ABG02]{ABG} B.~Allison, G.~Benkart, and Y.~Gao, \emph{Lie algebras graded by the 
root systems $BC_r$, $r\ge 2$},
Mem. Amer. Math. Soc. \textbf{751} (2002), x+158 pp.

\bibitem[Ara17]{Ara} D.~Aranda-Orna,  \emph{Fine gradings on simple exceptional Jordan pairs and 
triple systems}, 
J.~Algebra \textbf{491} (2017), 517--572.

\bibitem[AC21]{AC} D.~Aranda-Orna and A.S.~C\'ordova-Mart{\'\i}nez, \emph{Fine gradings on 
Kantor systems of Hurwitz type}, 
 Linear Algebra Appl. \textbf{613} (2021), 201--240.

\bibitem[BZ96]{BZ} G.~Benkart and E.~Zelmanov, \emph{Lie algebras graded by finite root systems 
and intersection matrix algebras}, 
Invent. Math. \textbf{126} (1996), 1--45.

\bibitem[BM92]{BermanMoody} S.~Berman and R.V.~Moody, \emph{Lie algebras graded by finite
 root systems and the intersection matrix algebras of Slodowy}, 
Invent. Math. \textbf{108} (1992), 323--347.

\bibitem[Bou02]{BouVI} N.~Bourbaki, \emph{Lie groups and Lie algebras. Chapters 4–6}, translated 
from the 1968 French original by Andrew Pressley, Elem. Math. (Berlin), Springer-Verlag, Berlin, 2002.
 xii+300 pp.

\bibitem[CE18]{CE} A.S.~C\'ordova-Mart{\'\i}nez and A.~Elduque, \emph{Gradings on 
semisimple algebras},
Linear Algebra Appl. \textbf{559} (2018), 145--171.

\bibitem[DET21]{DET} A.~Daza-Garc{\'\i}a, A.~Elduque, and L.~Tang, \emph{Cross 
products, automorphisms, and gradings}, 
Linear Algebra Appl. \textbf{610} (2021), 227--256.

\bibitem[DEM11]{DEM} C.~Draper, A.~Elduque, and C.~Mart{\'\i}n Gonz\'alez, 
\emph{Fine gradings on exceptional simple Lie superalgebras}, 
Internat. J. Math. \textbf{22} (2011), no.~12, 1823--1855.

\bibitem[Eld15]{Eld_root} A.~Elduque, \emph{Fine gradings and gradings by root systems 
on simple Lie algebras},
Rev. Mat. Iberoam. \textbf{31} (2015), no.~1, 245--266.

\bibitem[Eld16]{Eld16}
A.~Elduque, \emph{Gradings on algebras over algebraically closed fields}, in Non-associative 
and non-commutative algebra and operator theory, 113–121.
Springer Proc. Math. Stat., 160
Springer, Cham, 2016.

\bibitem[EK13]{EKmon} A.~Elduque and M.~Kochetov, \emph{Gradings on simple Lie algebras}, 
Mathematical Surveys and Monographs \textbf{189}, American Mathematical Society, Providence, RI;
Atlantic Association for Research in the Mathematical Sciences (AARMS), Halifax, NS, 2013.  xiv+336 pp.

\bibitem[EKR22]{EKR} A.~Elduque, M.~Kochetov, and A.~Rodrigo-Escudero, \emph{Gradings 
on associative algebras with involution and real forms of classical simple Lie algebras}, 
J.~Algebra \textbf{590} (2022), 61--138.


\bibitem[GOV94]{OVG} V.V.~Gorbatsevich, A.L.~Onishchik, and E.B.~Vinberg, \emph{Lie groups and 
Lie algebras III. Structure of Lie groups and Lie algebras}, translated from the Russian by V.~Minachin, 
Encyclopaedia of Mathematical Sciences, vol. 41, Springer- Verlag, Berlin, 1994. (Translation 
from ``Itogi Nauki i Tekhniki, Ser. Sovrem. Probl. Mat., Fundam. Napravleniya 41 (1990)''.)

\bibitem[Hes82]{Hesselink} W.H.~Hesselink, \emph{Special and pure gradings of Lie algebras},
Math. Z. \textbf{179} (1982), no.~1, 135--149.

\bibitem[Hum75]{Hum_lag} J.E.~Humphreys, \emph{Linear algebraic groups}, 
Graduate Texts in Mathematics \textbf{21}, Springer-Verlag, New York-Heidelberg, 1975. xiv+247 pp.

 



\bibitem[PZ89]{PZ} J.~Patera and H.~Zassenhaus, \emph{On Lie gradings. I}, Linear Algebra 
Appl. \textbf{112} (1989), 87--159.

\bibitem[SGA3]{SGA3} M.~Demazure, A.~Grothendieck, et al., \emph{S\'eminaire de 
g\'eom\'etrie alg\'ebrique: Sch\'emas en groupes}, 
Lecture Notes in Mathematics \textbf{151}, \textbf{152}, \textbf{153}, Springer, 
Berlin-Heidelberg-New York, 1970.

\bibitem[Ste68]{Steinberg}
R.~Steinberg, \emph{ Endomorphisms of linear algebraic groups},
Mem. Amer. Math. Soc., No. \textbf{80},
American Mathematical Society, Providence, RI, 1968, 108 pp.

\bibitem[Wat79]{Wat} W.C.~Waterhouse, \emph{Introduction to affine group schemes}, 
Graduate Texts in Mathematics \textbf{66}, Springer-Verlag, New York-Berlin, 1979. xi+164 pp.

\bibitem[Yu14]{YuS} J.~Yu, \emph{Maximal abelian subgroups of spin groups and some 
exceptional simple Lie groups}, arXiv:1403.2679.

\bibitem[Yu16]{YuE} J.~Yu, \emph{Maximal abelian subgroups of compact simple Lie groups 
of type E}, Geom. Dedicata \textbf{185} (2016), 205--269.

\end{thebibliography}
\end{document}